\documentclass[10pt,a4paper]{amsart}
\usepackage{amssymb}
\usepackage{xypic}
\newcommand{\Z}{{\mathbb Z}}
\newcommand{\F}{{\mathbb F}}
\renewcommand{\dim}{\operatorname{dim}\nolimits}

\renewcommand{\ker}{\operatorname{Ker}\nolimits}
\newcommand{\im}{\operatorname{Im}\nolimits}

\newtheorem{theorem}{Theorem}[section]
\newtheorem{proposition}[theorem]{Proposition}
\newtheorem{corollary}[theorem]{Corollary}
\newtheorem{lemma}[theorem]{Lemma}

\newtheorem{remark}[theorem]{Remark}
\newtheorem{example}[theorem]{Example}

\title[Spaces with Noetherian cohomology]{Spaces with Noetherian cohomology}
\author[K.K.S. Andersen]{Kasper K. S. Andersen}
\author[N. Castellana]{Nat\`{a}lia Castellana}
\author[V. Franjou]{Vincent Franjou}
\author[A. Jeanneret]{Alain Jeanneret}
\author[J. Scherer]{J\'er\^{o}me Scherer}
\thanks{The first author was supported by the Danish National Research Foundation 
through the Centre for Symmetry and Deformation, the second and fifth by the FEDER/MEC grant MTM2010-20692,
the third by the LMJL - Laboratoire de Math\'ematiques Jean-Leray UMR 6629 CNRS/Universit\'e de Nantes,
by the ANR grant HGRT BLAN08-2-338236, and by CRM Barcelona, and the fifth by the IHES and the University of Bern.}
\begin{document}

%%%%%%%%%%%%%%%%%%%%%%%%%%%%%%%%%%%%%%
\begin{abstract}
Is the cohomology of the classifying space of a $p$-compact group,
with Noetherian twisted coefficients, a Noetherian module? This note
provides, over the ring of $p$-adic integers, such a generalization
to $p$-compact groups of the Evens-Venkov Theorem. We consider the
cohomology of a space with coefficients in a module, and we compare
Noetherianity over the field with $p$ elements, with Noetherianity
over the $p$-adic integers, in the case when the fundamental group
is a finite $p$-group.
\end{abstract}
%%%%%%%%%%%%%%%%%%%%%%%%%%%%%%%%%%%%%%
\maketitle
%%%%%%%%%%%%%%%%%%%%%%%%%%%%%%%%%%%%%%%%%%%%%%%%%
\section*{Introduction}
\label{sec intro}
%%%%%%%%%%%%%%%%%%%%%%%%%%%%%%%%%%%%%%%%%%%%%%%%%
%%%%%%%%%%%%%%%%%%%%%%%%%%%%%%%%%%%%%%%%%%%%%%%%%
The main theorem of Dwyer and Wilkerson in \cite{MR1274096} states
that the mod~$p$ cohomology of the classifying space of a
$p$-compact group is a finitely generated algebra. This generalizes
to $p$-compact groups the Evens-Venkov Theorem \cite{MR0137742} on
the cohomology of a finite group $G$. There are however two main
differences between these two results. Evens' statements allow a
general base ring --- any Noetherian ring is allowed, and they
include the case of general twisted coefficients (contrary to the
early work by Golod, \cite{MR0104720}, or Venkov, \cite{MR0108788})
as follows: if $M$ is Noetherian over a ring $R$, then so is
$H^\ast(G;M)$ over $H^\ast(G;R)$. Beautiful finite generation
statements on cohomology have since been proved in numerous
situations. For statements as general as Evens' however, proofs have
been surprisingly elusive.
\par
This note is concerned with these generalizations for $p$-compact
groups and $p$-local finite groups, as defined by Broto, Levi, and Oliver,
\cite{MR1992826}. We ask more generally when
Noetherianity of the mod $p$
cohomology algebra $H^*(Y;\F_p )$ of a space $Y$ implies that the
cohomology with coefficients in a $R[\pi_1(Y)]$-module $M$, $H^*(Y;
M)$, is a Noetherian module over the algebra $H^*(Y; R)$. Because
the classifying space $BX$ of a $p$-compact group is $p$-complete by
definition, we work over $p$-complete
rings (for example $H^*((BS^3)^\wedge_p; \Z)$ is not Noetherian).
\par\medskip
%%%%%%%%%%%%%%%%%%%%%%%%%%%%%%%%%%%%%%%%%%%%%%%%%%%%%%%%%%
\noindent{\bf Theorem~\ref{thm:padicnoetherian}.}
%%%%%%%%%%%%%%%%%%%%%%%%%%%%%%%%%%%%%%%%%%%%%%%%%%%%%%%%%%
\emph{
Let $Y$ be a connected space with finite fundamental group. Then, the graded $\Z^\wedge_p$-algebra $H^*(Y; \Z^\wedge_p)$
is Noetherian if, and only if, the graded $\F_p$-algebra $H^*(Y;\F_p)$ is Noetherian and the torsion
in $H^*(Y; \Z^\wedge_p)$ is bounded.
}
\par\medskip
%%%%%%%%%%%%%%%%%%%%%%%%%%%%%%%%%%%%%%%%%%%%%%%%%%%%%%%%%%
\noindent{\bf Theorem~\ref{thm:main}.}
%%%%%%%%%%%%%%%%%%%%%%%%%%%%%%%%%%%%%%%%%%%%%%%%%%%%%%%%%%
\emph{
%\begin{main}
    Let $Y$ be a connected space such that $\pi_1 Y$ is a finite $p$-group.
    Let $M$ be a $\Z^\wedge_p [\pi_1 Y]$-module,
    which is finitely generated over $\Z^\wedge_p$.
If the graded $\Z^\wedge_p$-algebra $H^*(Y; \Z^\wedge_p)$ is Noetherian, then
$H^*(Y; M)$ is Noetherian as a module over $H^*(Y; \Z^\wedge_p)$.
%\end{main}
}
\par\medskip
This applies to $p$-compact group and to $p$-local finite groups
to show that their $p$-adic cohomology algebra is Noetherian,
see Theorem~\ref{thm:padicpcompact} and~\ref{thm:padicplfg}.
Note that our proof makes no use of the recent classification of
$p$-compact groups by Andersen, Grodal, M{\o}ller, and
Viruel, \cite{MR2476779}, \cite{MR2373153}, \cite{MR2338539}.
Even in the case of a compact Lie group $G$, our theorem provides a general
finiteness theorem for the cohomology of $BG$ with twisted coefficients.
One of the few explicit computations available in the literature is the case
of $O(n)$, due to \v{C}adek \cite{MR1709293}, (see also Greenblatt 
\cite{MR2246023}).
\par\medskip
\noindent {\bf Acknowledgements.} This work started when the third
author visited the CRM in Barcelona during the emphasis semester on
higher categories in 2008 and was continued while the fifth author
was visiting the IHES and the University of Bern in 2009. We would
like to thank these institutions for their generous hospitality. We would like to
thank Ran Levi for finding an extra author.
%Our warmest thanks go to the
%anonymous referee for his careful reading, improvements, and corrections, which
%led us to the present version of Theorem~\ref{thm:padicnoetherian}.

%%%%%%%%%%%%%%%%%%%%%%%%%%%%%%%%%%%%%%%%%%%%%%%%%%%%%%%%%%
%%%%%%%%%%%%%%%%%%%%%%%%%%%%%%%%%%%%%%%%%%%%%%%%%%%%%%%%%%
\section{The cohomology as a graded module}
\label{sec:gradedmodule}
%%%%%%%%%%%%%%%%%%%%%%%%%%%%%%%%%%%%%%%%%%%%%%%%%%%%%%%%%%
%%%%%%%%%%%%%%%%%%%%%%%%%%%%%%%%%%%%%%%%%%%%%%%%%%%%%%%%%%
Before considering the mod $p$ or $p$-adic cohomology as an algebra, we
first make explicit the relationship between two standard milder finiteness
assumptions. When the graded vector space $H^*(Y; \F_p)$ is of finite type,
i.e. $H^n(Y; \F_p)$ is a finite dimensional vector space in each degree $n$,
is $H^n(Y; \Z^\wedge_p)$ a finitely generated $\Z^\wedge_p$-module in each degree $n$ as well?
This is clearly a necessary condition for the cohomology algebra to be finitely generated.
We show that it holds when $\pi_1(Y)$ is finite.
\par
The main tool to relate the mod $p$ and the $p$-adic cohomology is the universal coefficient exact sequence
--- see for example \cite[Theorem~5.5.10]{MR666554} for spaces and \cite[Part III, Proposition 6.6]{Adams} for spectra:
\begin{equation}
\label{universal}
0 \rightarrow H^*(Y; \Z^\wedge_p) \otimes\F_p \xrightarrow{\rho} H^*(Y; \F_p)
\xrightarrow{\partial} \hbox{\rm Tor} (H^{*+1}(Y; \Z^\wedge_p); \Z/p) \rightarrow 0
\end{equation}
which applies, since $ \Z^\wedge_p$ is a PID  and $\Z/p$ is a finitely generated $\Z^\wedge_p$-module. 
%%%%%%%%%%%%%%%%%%%%%%%%%%%%%%%%%%%%%%%%%%%%%%%%%%%%%%%%%%
\begin{remark}
\label{universal coefficients}
%%%%%%%%%%%%%%%%%%%%%%%%%%%%%%%%%%%%%%%%%%%%%%%%%%%%%%%%%%
{\rm
The morphism $\rho$ in \eqref{universal} is a ring homomorphism
which makes the middle term $H^*(Y;\F_p)$ an $H^*(Y; \Z^\wedge_p) \otimes \F_p$-module.
Evens observed in \cite[p. 272]{MR0430024} that $\partial$ is also a homomorphism of
$H^*(Y;\Z^\wedge_p) \otimes \F_p$-modules,
where $\hbox{\rm Tor} (H^*(Y; \Z^\wedge_p); \Z/p)$ has the natural
module structure he introduced  in \cite[Lemma~2]{MR0430024}.}
\end{remark}
%
%%%%%%%%%%%%%%%%%%%%%%%%%%%%%%%%%%%%%%%%%%%%%%%%%%%%%%%%%%
\begin{lemma}
\label{lem:finiteinvariants}
%%%%%%%%%%%%%%%%%%%%%%%%%%%%%%%%%%%%%%%%%%%%%%%%%%%%%%%%%%
Let $G$ be a finite $p$-group, $K$ a field of characteristic $p$ and $V$ a $KG$-module. If $V^G$ is 
a finite dimensional $K$-vector space, then so is $V$. 
\end{lemma}
\begin{proof}
Let $n=\dim_K V^G$ and let $F=(KG)^n$ be a free $KG$-module of
rank $n$. Note that $\dim_K F^G= n$, so there is an isomorphism of
$KG$-modules $\alpha \colon V^G\rightarrow F^G$. Since $F$ is an
injective $KG$-module, $\alpha$ extends to a homomorphism $\alpha'
\colon V \rightarrow F$ of $KG$-modules, which we now prove is injective. 
Clearly $(\ker\alpha')^G= \ker\alpha'\cap V^G=\ker\alpha=0$. Since $G$ is a 
finite $p$-group, it follows that $\ker\alpha'= 0$.
Hence $V$ embeds in $F$, so $V$ is finite dimensional.
\end{proof}
%

%%%%%%%%%%%%%%%%%%%%%%%%%%%%%%%%%%%%%%%%%%%%%%%%%%%%%%%%%%
\begin{proposition}
\label{padicfinitetype}
%%%%%%%%%%%%%%%%%%%%%%%%%%%%%%%%%%%%%%%%%%%%%%%%%%%%%%%%%%
Let $Y$ be a connected space with finite fundamental group. The group
$H^n(Y;\F_p)$ is finite for every positive integer $n$ if and only if
the $\Z^\wedge_p$-module $H^n(Y;\Z^\wedge_p)$ is finitely generated for every $n$.
Under this condition, the $\Z^\wedge_p$-module $H^n(Y; M)$ is finitely generated
for any $n$ and every $\Z^\wedge_p[\pi_1 Y]$-module $M$ which is finitely generated over $\Z^\wedge_p$.
\end{proposition}
\begin{proof}
If $H^n(Y; \Z^\wedge_p)$ is a finitely generated $\Z^\wedge_p$-module for any $n$,
the universal coefficient exact sequence \eqref{universal} implies that $H^n(Y;\F_p)$ is
finite for any $n$.
\par
Conversely, assume that $H^n(Y;\F_p)$ is finite for every $n$. Since
the fundamental group of $Y$ is finite, the space $Y$ is $p$-good by
\cite[Proposition~VII.5.1]{MR51:1825} and therefore
$H^n(Y^\wedge_p;\F_p)\cong H^n(Y;\F_p)$. Likewise, since cohomology
with $p$-adic coefficients is represented by Eilenberg-MacLane
spaces $K(\Z^\wedge_p, n)$, which are $p$-complete,
$H^n(Y^\wedge_p;\Z^\wedge_p)\cong H^n(Y;\Z^\wedge_p)$, \cite[Proposition~II.2.8]{MR51:1825}.
We may therefore assume that $Y$ is $p$-complete and
that $G=\pi_1 Y$ is a finite $p$-group, see
\cite[Proposition 11.14]{MR1274096} or \cite[Section~5]{MR1062866}.
\par
If $Y$ is $1$-connected, then \cite[Proposition 5.7]{ABN} applies
and $H^n(Y;\Z^\wedge_p)$ is a finitely generated $\Z^\wedge_p$-module for every $n$.
For the general situation, let us
consider the universal cover fibration for $Y$, $\tilde Y \rightarrow Y
\rightarrow BG$. We prove by
induction that $H^n(\tilde Y; \F_p)$ is finite dimensional for any $n$.
The induction starts with the trivial case $n= 0$.
Assume thus that $H^m(\tilde Y; \F_p)$ is finite for all $m< n$.
Then, in the second page of the Serre spectral sequence in mod $p$ cohomology,
all groups $E_2^{i, j} = H^i(BG,H^j(\tilde Y; \F_p))$ on the lines $j=0, \dots, n-1$ are finite.
As $E_\infty^{0, n}$ is finite as well, it follows
that $E_2^{0, n} = H^n(\tilde Y; \F_p)^G$ is finite dimensional.
Since $G$ is a finite $p$-group,  finiteness of
$H^n(\tilde Y; \F_p)^G$ implies finiteness of $H^n(\tilde Y; \F_p)$ by Lemma \ref{lem:finiteinvariants}.
\par
We can now apply the 1-connected case to conclude that
$H^n(\tilde Y; \Z^\wedge_p)$ is a finitely generated $\Z^\wedge_p$-module for any $n$.
The Evens-Venkov Theorem
\cite[Theorem 8.1]{MR0137742} now shows that the $E_2$-term of the
Serre spectral sequence with $p$-adic coefficients consists of
finitely generated $\Z^\wedge_p$-modules. Thus so must be
$H^n(Y; \Z^\wedge_p)$ for any $n$.
\par
The second part of the assertion now follows easily.
The first part of the proposition and the universal coefficient
formula imply that $H^n(\tilde Y; M)$ is a finitely generated
$\Z^\wedge_p$-module for every~$n$.
We then use the Serre spectral sequence for cohomology with twisted
coefficients. The only reference we know is
\cite[Theorem~3.2]{MR1123662} where it is done equivariantly; we
need the case of the trivial group action.
\end{proof}

%%%%%%%%%%%%%%%%%%%%%%%%%%%%%%%%%%%%%%%%%%%%%%%%%%%%%%%%%%
%%%%%%%%%%%%%%%%%%%%%%%%%%%%%%%%%%%%%%%%%%%%%%%%%%%%%%%%%%
\section{Cohomology with Trivial coefficients}
\label{sec:trivial}
%%%%%%%%%%%%%%%%%%%%%%%%%%%%%%%%%%%%%%%%%%%%%%%%%%%%%%%%%%
%%%%%%%%%%%%%%%%%%%%%%%%%%%%%%%%%%%%%%%%%%%%%%%%%%%%%%%%%%
We now turn to finite generation of the cohomology algebras
$H^*(Y;\Z^\wedge_p )$ and $H^*(Y;\F_p )$, where trivial
coefficients are understood. This section should thus be no more than a warm up,
because it seems enough to gain some control on torsion to draw conclusion
from the universal coefficient theorem.
\par
Let $R$ be either the ring $ \Z^\wedge_p$, or the field $ \F_p$,
and note that both are Noetherian rings. The cohomology $H^*(Y;R )$
of any connected space is a commutative graded algebra, which is a
Noetherian $R$-algebra if and only if it is finitely generated as an
$R$-algebra \cite[Theorem 13.1]{MR1011461}.
%%%%%%%%%%%%%%%%%%%%%%%%%%%%%%%%%%%%%%%%%%%%%%%%%%%%%%%%%%
\begin{lemma}
\label{lem:finitetypemodp}
%%%%%%%%%%%%%%%%%%%%%%%%%%%%%%%%%%%%%%%%%%%%%%%%%%%%%%%%%%
Let $Y$ be a connected space.
If the $\Z^\wedge_p$-algebra $H^*(Y; \Z^\wedge_p)$ is Noetherian,
then $H^*(Y; \F_p)$ is a finitely generated module over the algebra
$H^*(Y; \Z^\wedge_p)\otimes \F_p$.
\end{lemma}
\begin{proof}
The ideal $\hbox{\rm Tor} (H^*(Y; \Z^\wedge_p); \Z/p)$ of elements annihilated by $p$
is a finitely generated ideal of $H^*(Y; \Z^\wedge_p)$ by assumption. It is therefore also finitely generated as an
$H^*(Y; \Z^\wedge_p)\otimes \F_p$-module.
The conclusion follows from Remark~\ref{universal coefficients} on the universal coefficient exact sequence.
\end{proof}

To be able to compare Noetherianity of the mod $p$ and the $p$-adic cohomology,
we need to analyze the $p$-torsion in $H^*(Y; \Z^\wedge_p)$.
Let us denote by $\text{T}_pH^*(Y; \Z^\wedge_p)$ the graded submodule of $p$-torsion elements.
The key assumption in the main theorem of this section is that the order of the $p$-torsion is bounded.
This implies that $\rho$ is ``uniformly power surjective", a strong form of integrality.

%%%%%%%%%%%%%%%%%%%%%%%%%%%%%%%%%%%%%%%%%%%%%%%%%%%%%%%%%%
\begin{lemma}
\label{lem:Bockstein}
%%%%%%%%%%%%%%%%%%%%%%%%%%%%%%%%%%%%%%%%%%%%%%%%%%%%%%%%%%
Let $Y$ be a connected space and let $d$ be an integer such that
$p^d \cdot  \hbox{\rm T}_p H^*(Y; \Z^\wedge_p) = 0$.
%Assume that $H^n(Y; \Z^\wedge_p)$ is a finitely
%generated $\Z^\wedge_p$-module for any $n$.
If $u \in H^*(Y; \F_p)$, then $u^{p^d}$ belongs to the image of
$\rho: H^*(Y; \Z^\wedge_p) \otimes\F_p \rightarrow H^*(Y; \F_p)$.
\end{lemma}

\begin{proof}
Following the elementary proof of \cite[Lemma~4.4]{MR885093},
we start with the observation that for any element $x \in H^*(Y; \Z/p^k)$ the $p$-th power
$x^p$ lies in the image of the reduction map $H^*(Y; \Z/p^{k+1}) \rightarrow H^*(Y; \Z/p^k)$.
The argument is as follows: If $p$ is odd and the degree of $x$ is odd,
$x^p=0$ and the conclusion follows. Otherwise, $\delta(x^p) = p \delta(x) \cdot x^{p-1} = 0$,
because the Bockstein $\delta$ coming from the short exact sequence
$\Z/p \rightarrow \Z/p^{k+1} \rightarrow \Z/p^k$ is a derivation with respect to the cup product pairing $H^*(Y;\Z/p^k)\otimes H^*(Y;\Z/p)\rightarrow H^*(Y;\Z/p)$.
Therefore $u^{p^d}$ lies in the image of the reduction $H^*(Y; \Z/p^{d+1}) \rightarrow H^*(Y; \F_p)$.
\par
The diagram of short exact sequences
\[
%\begin{array}{c}
\xymatrix{ 0\ar[r] &  \Z \ar[d]^{\cdot p^d} \ar[r]^{\cdot p^{d+1}} &
\Z \ar@{=}[d]\ar[r] & \Z/p^{d+1}
\ar@{->>}[d]\ar[r] &  0
\\
0\ar[r] & \Z \ar[r]^-{\cdot p} &
\Z\ar[r] & \Z/p \ar[r] &  0}
%\end{array}
\]
induces the commutative diagram of exact rows:
\[
%\begin{array}{c}
\xymatrix{ 0\to   \hbox{\rm Tor} (H^{*+1}(Y; \Z^\wedge_p); \Z/p^{d+1}) \ar[d] \ar[r] &
H^{*+1}(Y; \Z^\wedge_p) \ar[d]^{\cdot p^d}\ar[r]^{\cdot p^{d+1}} & H^{*+1}(Y; \Z^\wedge_p)\ar@{=}[d]
\\
0\to   \hbox{\rm Tor} (H^{*+1}(Y; \Z^\wedge_p); \Z/p) \ar[r] &
H^{*+1}(Y; \Z^\wedge_p) \ar[r]^{\cdot p} & H^{*+1}(Y; \Z^\wedge_p)}%\end{array}
\]
Since $p^d \cdot  \hbox{\rm T}_p H^*(Y; \Z^\wedge_p) = 0$, the left vertical morphism is zero.
Consider now the two universal coefficient sequences
relating the cohomology of $Y$
with coefficients in $\F_p$, respectively in $\Z/p^{d+1}$,
to the cohomology of $Y$ with coefficients in $\Z^\wedge_p$ :
\[
%\begin{array}{c}
\xymatrix{ 0\to  H^*(Y; \Z^\wedge_p)\otimes \Z/p^{d+1} \ar[d] \ar[r] &
H^*(Y; \Z/p^{d+1}) \ar[d]\ar[r] & \hbox{\rm Tor} (H^{*+1}(Y; \Z^\wedge_p); \Z/p^{d+1})
\ar[d]\to  0
\\
0\to  H^*(Y; \Z^\wedge_p)\otimes \F_p \ar[r]^-{\rho} &
H^*(Y; \F_p) \ar[r]^-{\partial} & \hbox{\rm Tor} (H^{*+1}(Y; \Z^\wedge_p); \Z/p) \to  0}
%\end{array}
\]
where the vertical morphisms are induced by the mod $p$ reduction $\Z/p^{d+1} \to \Z/p$. The element
$u^{p^d}$ lies in the image of the mod $p$ reduction and we have shown that the morphism between the
torsion groups on the right is zero. Therefore $\partial(u^{p^d})=0$, which implies that $u^{p^d}$ is in $\im\rho$.
\end{proof} 

%%%%%%%%%%%%%%%%%%%%%%%%%%%%%%%%%%%%%%%%%%%%%%%%%%%%%%%%%%
\begin{lemma}
\label{prop:finitetype2}
%%%%%%%%%%%%%%%%%%%%%%%%%%%%%%%%%%%%%%%%%%%%%%%%%%%%%%%%%%
Let $Y$ be a connected space. If the graded $\F_p$-algebra $H^*(Y; \F_p)$ is Noetherian
and if $H^*(Y; \Z^\wedge_p)$ has bounded torsion,
then $H^*(Y; \F_p)$ is a finitely generated module over $H^*(Y; \Z^\wedge_p)\otimes \F_p$.
\end{lemma}
\begin{proof}
This is clear since Lemma \ref{lem:Bockstein} implies that $H^*(Y; \F_p)$ is integral over $\im\rho$.
Explicitely, let us choose homogeneous generators $w_1, \cdots, w_n$ of the graded
algebra $H^*(Y; \F_p)$ and consider the finite set $W$ of monomials of the
form $w_1^{r_1} \cdots w_n^{r_n}$ with $0 \leq r_i < p^d$.
We show that the set $W$ generates $H^*(Y; \F_p)$ as a module over
$H^*(Y; \Z^\wedge_p)\otimes \F_p$.
For, consider any monomial $m = w_1^{s_1} \cdots w_n^{s_n}$ in
$H^*(Y; \F_p)$. Writing the exponents $s_i = r_i + p^d t_i$
with $0 \leq r_i < p^d$, we express
$m=x^{p^d}\cdot w$ for a monomial $w$ in $W$ and an homogeneous element $x$.
By Lemma \ref{lem:Bockstein}, $x^{p^d}$ lifts to an element $a$ in $H^*(Y; \Z^\wedge_p)\otimes \F_p$
and $m=\rho(a)\cdot w$.
\end{proof}
%%%%%%%%%%%%%%%%%%%%%%%%%%%%%%%%%%%%%%%%%%%%%%%%%%%%%%%%%%
\begin{theorem}
\label{thm:padicnoetherian}
%%%%%%%%%%%%%%%%%%%%%%%%%%%%%%%%%%%%%%%%%%%%%%%%%%%%%%%%%%
Let $Y$ be a connected space with finite fundamental group. Then, the graded $\Z^\wedge_p$-algebra $H^*(Y; \Z^\wedge_p)$ is Noetherian
if and only if
the graded $\F_p$-algebra $H^*(Y;\F_p)$ is Noetherian and the torsion in $H^*(Y; \Z^\wedge_p)$ is bounded.
\end{theorem}
\begin{proof}
Assume first that $H^*(Y; \Z^\wedge_p)$ is a Noetherian
$\Z^\wedge_p$-algebra. By Lemma \ref{lem:finitetypemodp},
$H^*(Y;\F_p)$ is a finitely generated module over $H^*(Y; \Z^\wedge_p)\otimes \F_p$.
Since $H^*(Y; \Z^\wedge_p)\otimes \F_p$ is a Noetherian $\F_p$-algebra,
it follows from \cite[Proposition 7.2]{MR0242802} that
$H^*(Y;\F_p)$ is also a Noetherian $\F_p$-algebra. The torsion part $\text{T}_p H^*(Y; \Z^\wedge_p)$
is an ideal of the Noetherian algebra $H^*(Y; \Z^\wedge_p)$, hence is finitely generated.
The order of the torsion is thus bounded by the order of
its generators.
\par
Suppose now that $H^*(Y;\F_p)$ is a Noetherian $\F_p$-algebra and that
the torsion in $H^*(Y; \Z^\wedge_p)$ is bounded. Then, by
Lemma~\ref{prop:finitetype2}, $H^*(Y;\F_p)$
is a finitely generated module over
$H^*(Y; \Z^\wedge_p)\otimes \F_p$. As a consequence of
the graded version of the so-called Eakin-Nagata Theorem, see
Proposition~\ref{gradedEakinNagata}, we infer
then that the graded subring $H^*(Y; \Z^\wedge_p)\otimes \F_p$ of
$H^*(Y;\F_p)$ is also Noetherian. Since $H^*(Y;\F_p)$ is finitely generated, Proposition~\ref{padicfinitetype} shows 
that $H^n(Y;\Z^\wedge_p)$ is a finitely generated $\Z^\wedge_p$-module, hence Hausdorff, in each degree. 
Thus $H^*(Y; \Z^\wedge_p)$ is a Noetherian $\Z^\wedge_p$-algebra by Corollary~\ref{cor:Hausdorff-algebra}.
\end{proof}

We end this section with an example which shows that
Theorem~\ref{thm:padicnoetherian} does not hold
without the condition on torsion.
%%%%%%%%%%%%%%%%%%%%%%%%%%%%%%%%%%%%%%%%%%%%%%%%%%%%%%%%%%
\begin{example}
\label{ex:ABN}
%%%%%%%%%%%%%%%%%%%%%%%%%%%%%%%%%%%%%%%%%%%%%%%%%%%%%%%%%%
{\rm
Aguad\'e, Broto, and Notbohm constructed in \cite{ABN} spaces $X_k(r)$ for any odd
prime $p$ with $r|p-1$ and $k\geq 0$ satisfying :
$H^*(X_k(r); \F_p) \cong \F_p[x_{2r}] \otimes E(\beta^{(k+1)} x_{2r})$
where $\beta^{(k+1)}$ denotes the Bockstein of order $k+1$.
Observe that
$H^*(X_k(r); \F_p)$ is a Noetherian $\F_p$-algebra.
The torsion of $H^*(X_k(r); \Z^\wedge_p)$ is unbounded by \cite[Remark~5.8]{ABN}.
Theorem~\ref{thm:padicnoetherian} shows that the algebra
$H^*(X_k(r); \Z^\wedge_p)$ is not Noetherian.
}
\end{example}

%%%%%%%%%%%%%%%%%%%%%%%%%%%%%%%%%%%%%%%%%%%%%%%%%%%%%%%%%%
%%%%%%%%%%%%%%%%%%%%%%%%%%%%%%%%%%%%%%%%%%%%%%%%%%%%%%%%%%
\section{Cohomology with twisted coefficients}
\label{sec:twisted}
%%%%%%%%%%%%%%%%%%%%%%%%%%%%%%%%%%%%%%%%%%%%%%%%%%%%%%%%%%
%%%%%%%%%%%%%%%%%%%%%%%%%%%%%%%%%%%%%%%%%%%%%%%%%%%%%%%%%%
In this section we work over a ring $R$ which is either $\Z^\wedge_p
$ or $ \F_p$. Let $Y$ be a connected space whose fundamental
group is a finite $p$-group. Let $M$ be a $ R[\pi_1 Y]$-module which is a
finitely generated $R$-module. We aim to show that the cohomology
with twisted coefficients $H^*(Y; M)$ is Noetherian as a module over
$H^*(Y; R)$ if $H^*(Y, R)$ is Noetherian.
We shall deal separately with the field of $p$ elements and with the ring of
$p$-adic integers.
\par
We start with a standard Noetherianity result.
%%%%%%%%%%%%%%%%%%%%%%%%%%%%%%%%%%%%%%%%%%%%%%%%%%%%%%%%%%
\begin{lemma}
\label{lem:shortexact}
%%%%%%%%%%%%%%%%%%%%%%%%%%%%%%%%%%%%%%%%%%%%%%%%%%%%%%%%%%
Let $R=\Z^\wedge_p $ or $ \F_p$. Let $Y$ be a space and let
$0 \rightarrow N \rightarrow M \rightarrow Q\rightarrow 0$
be a short exact sequence of $R[\pi_1 Y]$-modules.
If both $H^*(Y; N)$ and $H^*(Y; Q)$ are Noetherian modules over $H^*(Y;R)$,
then so is $H^*(Y; M)$.
\end{lemma}
\begin{proof}
The long exact sequence in cohomology induced by
the short exact sequence of modules is one of $H^*(Y; R)$-modules. It
exhibits $H^*(Y; M)$ as an extension of a submodule of $H^*(Y; Q)$
by a quotient of $H^*(Y;N)$.
\end{proof}
%%%%%%%%%%%%%%%%%%%%%%%%%%%%%%%%%%%%%%%%%%%%%%%%%%%%%%%%%%
\subsection{The case of $\F_p$-vector spaces}
\label{sec:modp}
%%%%%%%%%%%%%%%%%%%%%%%%%%%%%%%%%%%%%%%%%%%%%%%%%%%%%%%%%%
To prove the next result we follow Minh and Symonds' approach  for profinite groups,
\cite[Lemma~1]{MinhSymondsPreprint}.
%%%%%%%%%%%%%%%%%%%%%%%%%%%%%%%%%%%%%%%%%%%%%%%%%%%%%%%%%%
\begin{theorem}
\label{thm:Fpmodule}
%%%%%%%%%%%%%%%%%%%%%%%%%%%%%%%%%%%%%%%%%%%%%%%%%%%%%%%%%%
Let $Y$ be a connected space such that $\pi_1 Y$ is a finite
$p$-group and let $M$ be a finite $\F_p[\pi_1 Y]$-module. If the
graded $\F_p$-algebra $H^*(Y; \F_p)$ is Noetherian, then $H^*(Y; M)$
is Noetherian as a module over $H^*(Y; \F_p)$.
\end{theorem}
\begin{proof}
We use induction on $\dim_{\F_p} M$. Since $G=\pi_1(Y)$ is a
finite $p$-group, the invariant submodule $M^G$ is not trivial when $M$ is not trivial.
The induction step follows by applying Lemma~\ref{lem:shortexact} to the short exact
sequence $0 \rightarrow M^G \rightarrow M \rightarrow M/M^G \rightarrow 0$.
\end{proof}
%%%%%%%%%%%%%%%%%%%%%%%%%%%%%%%%%%%%%%%%%%%%%%%%%
\subsection{The case of $\Z^\wedge_p$-modules}
\label{sec:torsion}
%%%%%%%%%%%%%%%%%%%%%%%%%%%%%%%%%%%%%%%%%%%%%%%%%
We consider in this section the cohomology with twisted coefficients
$H^*(Y; M)$ of a connected space $Y$ where $M$ is a $\Z^\wedge_p
[\pi_1 Y]$-module which is finitely generated over~$\Z^\wedge_p$. In
a first step, let $M$ be a $\Z^\wedge_p [\pi_1 Y]$-module which is
finite (meaning finite as a set).
%%%%%%%%%%%%%%%%%%%%%%%%%%%%%%%%%%%%%%%%%%%%%%%%%%%%%%%%%%
\begin{lemma}
\label{lem:torsionmodule}
%%%%%%%%%%%%%%%%%%%%%%%%%%%%%%%%%%%%%%%%%%%%%%%%%%%%%%%%%%
Let $Y$ be a connected space such that $\pi_1 Y$ is a finite
$p$-group. Let $M$ be a $\Z^\wedge_p [\pi_1 Y]$-module which is
finite. If the graded $\Z^\wedge_p$-algebra $H^*(Y; \Z^\wedge_p)$ is Noetherian,
then $H^*(Y; M)$ is Noetherian as a module over
$H^*(Y;\Z^\wedge_p)$.
\end{lemma}
\begin{proof}
The module $M$ being finite, $M$ is a finite abelian $p$-group.
We perform an induction on the exponent $e$ of $M$. When $e=1$, the
module $M$ has the structure of an $\F_p$-vector space. As $H^*(Y; \F_p)$ is
a Noetherian $\F_p$-algebra by Theorem~\ref{thm:padicnoetherian},
we know from Theorem~\ref{thm:Fpmodule} that $H^*(Y; M)$ is
Noetherian as a module over $H^*(Y; \F_p)$.
The Noetherian $\Z^\wedge_p$-algebra $H^*(Y; \Z^\wedge_p)$ acts
on $H^*(Y; M)$ through \hbox{$H^*(Y; \Z^\wedge_p) \otimes \F_p$}. By Lemma \ref{lem:finitetypemodp}, $H^*(Y;\F_p)$ is finitely generated as a $H^*(Y;\Z^\wedge_p)$-module. Therefore $H^*(Y; M)$ is a Noetherian module over $H^*(Y; \Z^\wedge_p)$.
\par
Let us now assume that $e > 1$ and consider the short exact sequence
$0\rightarrow M_p \rightarrow M \rightarrow Q \rightarrow 0$ where
$M_p$ is the submodule of $M$ consisting of elements of order 1 or~$p$.
The induction step follows from Lemma~\ref{lem:shortexact}.
\end{proof}
%%%%%%%%%%%%%%%%%%%%%%%%%%%%%%%%%%%%%%%%%%%%%%%%%%%%%%%%%%
\begin{remark}
\label{rem:nouniversal}
%%%%%%%%%%%%%%%%%%%%%%%%%%%%%%%%%%%%%%%%%%%%%%%%%%%%%%%%%%
{\rm In the case of trivial coefficient modules our main tool was
the universal coefficient exact sequence, but this does not exist
in general for twisted coefficients. One basic counter
example is given by the module $M=\F_p[G]$ for a finite group $G$ whose
order is divisible
by $p$. Then $H^*(BG;M)$ is zero in positive degrees and the universal coefficient
formula does not hold.}
\end{remark}
In a second step we consider, as coefficient of the cohomology, a
$\Z^\wedge_p [\pi_1 Y]$-module $M$, which is free of finite
rank over $\Z^\wedge_p$.
%%%%%%%%%%%%%%%%%%%%%%%%%%%%%%%%%%%%%%%%%%%%%%%%%%%%%%%%%%
\begin{lemma}
\label{lem:torsionfreemodule}
%%%%%%%%%%%%%%%%%%%%%%%%%%%%%%%%%%%%%%%%%%%%%%%%%%%%%%%%%%
Let $Y$ be a connected space such that $\pi_1 Y$ is a finite
$p$-group. Let $M$ be a $\Z^\wedge_p [\pi_1 Y]$-module which is free
of finite rank over $\Z^\wedge_p$. If the graded $\Z^\wedge_p$-algebra
$H^*(Y; \Z^\wedge_p)$ is Noetherian, then $H^*(Y; M)$ is Noetherian as a module
over $H^*(Y; \Z^\wedge_p)$.
\end{lemma}
\begin{proof}
The short exact sequence $0 \rightarrow M \xrightarrow{\cdot p} M
\rightarrow M \otimes \F_p \rightarrow 0$ induces in
cohomology a long exact sequence of $H^*(Y; \Z^\wedge_p)$-modules.
We see that $H^*(Y; M) \otimes \F_p$ is
a sub-$H^*(Y; \Z^\wedge_p)$-module of $H^*(Y; M \otimes\F_p)$.
Since the action of $H^*(Y; \Z^\wedge_p)$ on
both $H^*(Y; M \otimes \F_p)$ and $H^*(Y; M )\otimes \F_p$ factors through
$H^*(Y; \Z^\wedge_p)\otimes \F_p$, it follows that $H^*(Y; M) \otimes \F_p$ is a
sub-$H^*(Y; \Z^\wedge_p)\otimes \F_p$-module of $H^*(Y; M \otimes \F_p)$.
\par
This takes us back to the world of $\F_p$-vector spaces. We know by Theorem~\ref{thm:Fpmodule} 
that $H^*(Y; M \otimes\F_p)$ is a Noetherian module over $H^*(Y; \F_p)$, a Noetherian
algebra by Theorem~\ref{thm:padicnoetherian}.
As the latter is a finitely generated module over $H^*(Y; \Z^\wedge_p)\otimes \F_p$
by Lemma~\ref{lem:finitetypemodp},
we infer that $H^*(Y; M \otimes \F_p)$ is a Noetherian module
over $H^*(Y; \Z^\wedge_p)\otimes \F_p$. Therefore
$H^*(Y; M)\otimes \F_p$ is a Noetherian module over
$H^*(Y;\Z^\wedge_p)\otimes \F_p$ as well,
and since $H^*(Y;\Z^\wedge_p)$ acts on $H^*(Y; M)\otimes \F_p$ via
$H^*(Y; \Z^\wedge_p)\otimes \F_p$, it is a Noetherian module
over $H^*(Y; \Z^\wedge_p)$.
\par
Set $A^* =H^*(Y; \Z^\wedge_p)$ and $N^* = H^*(Y; M)$. Both are
finitely generated $\Z^\wedge_p$-modules in each degree by
Proposition \ref{padicfinitetype}, thus also
Hausdorff and complete. We then conclude by applying
Proposition~\ref{prop:Hausdorff}.
\end{proof}
We now prove our main theorem.
%%%%%%%%%%%%%%%%%%%%%%%%%%%%%%%%%%%%%%%%%%%%%%%%%%%%%%%%%%
\begin{theorem}
\label{thm:main}
%%%%%%%%%%%%%%%%%%%%%%%%%%%%%%%%%%%%%%%%%%%%%%%%%%%%%%%%%%
Let $Y$ be a connected space such that $\pi_1 Y$ is a finite
$p$-group. Let $M$ be a $\Z^\wedge_p [\pi_1 Y]$-module, which is
finitely generated over $\Z^\wedge_p$. If the graded $\Z^\wedge_p$-algebra
$H^*(Y; \Z^\wedge_p)$ is Noetherian, then $H^*(Y; M)$ is Noetherian as a
module over $H^*(Y; \Z^\wedge_p)$.
\end{theorem}
\begin{proof}
Let $TM$ be the torsion submodule of $M$ and consider the short
exact sequence of $\Z^\wedge_p [\pi_1 Y]$-modules $0
\rightarrow TM \rightarrow M \rightarrow Q \rightarrow 0$. We know
from Lemma~\ref{lem:torsionmodule} that $H^*(Y; TM)$ is a
Noetherian $H^*(Y; \Z^\wedge_p)$-module and from
Lemma~\ref{lem:torsionfreemodule} that so is $H^*(Y; Q)$. We
conclude by Lemma~\ref{lem:shortexact}.
\end{proof}
%%%%%%%%%%%%%%%%%%%%%%%%%%%%%%%%%%%%%%%%%%%%%%%%%%%%%%%%%%
\begin{remark}
\label{rem:Natalia}
%%%%%%%%%%%%%%%%%%%%%%%%%%%%%%%%%%%%%%%%%%%%%%%%%%%%%%%%%%
{\rm Our main theorem makes no assumption,
except that the fundamental group be a $p$-group.
One could try to relax it with transfer arguments,
requiring a version of the transfer with twisted coefficients.
However, recent work of Levi and Ragnarsson, in the context of $p$-local
finite group theory, provides \cite[Proposition 3.1]{Levi-Ragnarsson}
an example showing that such a transfer might not have, in general,
the properties we need when the fundamental group of the space is not a $p$-group.}
\end{remark}

%%%%%%%%%%%%%%%%%%%%%%%%%%%%%%%%%%%%%%%%%%%%%%%%%
%%%%%%%%%%%%%%%%%%%%%%%%%%%%%%%%%%%%%%%%%%%%%%%%%
\section{The case of $p$-compact groups and $p$-local finite groups}
\label{sec:pcg}
%%%%%%%%%%%%%%%%%%%%%%%%%%%%%%%%%%%%%%%%%%%%%%%%%
%%%%%%%%%%%%%%%%%%%%%%%%%%%%%%%%%%%%%%%%%%%%%%%%%
We arrive at the promised application to $p$-compact groups and $p$-local finite groups. By
definition, a $p$-compact group is a mod $p$ finite loop space $X = \Omega BX$, where the
``classifying space" $BX$ is $p$-complete, \cite{MR1274096}.

%%%%%%%%%%%%%%%%%%%%%%%%%%%%%%%%%%%%%%%%%%%%%%%%%%%%%%%%%%
\begin{lemma}
\label{lem:pcgbounded}
%%%%%%%%%%%%%%%%%%%%%%%%%%%%%%%%%%%%%%%%%%%%%%%%%%%%%%%%%%
Let $X$ be a $p$-compact group. Then the $p$-torsion in $H^*(BX; \Z^\wedge_p)$
is bounded.
\end{lemma}

\begin{proof}
By \cite[Proposition 9.9]{MR1274096}, any $p$-compact group admits a
maximal toral $p$-compact subgroup  $S$ such that $\iota\colon BS\rightarrow BX$
is a monomorphism and the Euler characteristic $\chi$ of the homotopy
fibre is prime to $p$ (see \cite[Proof of 2.4, page 431]{MR1274096}).  The Euler characteristic is the alternating sum 
of the ranks of the $\F_p$-homology  groups. Dwyer constructed a transfer map $\tau \colon \Sigma^\infty BX \rightarrow \Sigma^\infty BS$
in \cite{MR1384465} such that $\iota \circ \tau$ induces multiplication by $\chi$ on mod $p$ cohomology.
This is an isomorphism, so that the homotopy cofiber $C$ of $\iota \circ \tau$ has trivial mod $p$ cohomology.
\par
Moreover both $BX$ and $BS$ have finite mod $p$ cohomology in each degree and 
finite fundamental group, \cite[Lemma 2.1]{MR1274096}. Proposition~\ref{padicfinitetype} 
applies and in any degree, the $p$-adic cohomology modules of $BX$ and $BS$ are finitely
generated over $\Z^\wedge_p$.
The long exact sequence in cohomology associated to a cofibration then shows that 
the $\Z^\wedge_p$-modules $H^n(C; \Z^\wedge_p)$ are finitely generated for all $n$.
Since $H^*(C; \F_p)$ is trivial, it follows from the universal coefficient exact sequence 
\eqref{universal} that $H^*(C; \Z^\wedge_p)\otimes \F_p$ is trivial as well.
We conclude, by the Nakayama lemma, that $H^*(C; \Z^\wedge_p)$ is trivial,
i.e. $\iota \circ \tau$ induces also an isomorphism in  cohomology
with $p$-adic coefficients.
Therefore $\iota^*\colon H^*(BX;\Z^\wedge_p)\rightarrow H^*(BS;\Z^\wedge_p)$ is a monomorphism.
We are reduced to show that $H^*(BS;\Z^\wedge_p)$ has bounded torsion.
\par
Now, a toral $p$-compact group $S$ can be constructed, up to $p$-completion,
as an extension of a finite $p$-group $P$ and a discrete torus $H=\bigoplus \Z_{p^\infty}$. 
The fibration $BH^\wedge_p\simeq K(\bigoplus \Z^\wedge_p, 2)\rightarrow BS \rightarrow BP$
yields a finite covering $BH^\wedge_p\rightarrow BS$  and a classical transfer argument shows then 
that multiplication by $|P|$ on $H^*(BS; \Z^\wedge_p)$ factors through the torsion free module 
$H^*(BH^\wedge_p; \Z^\wedge_p)$.
\end{proof}

%%%%%%%%%%%%%%%%%%%%%%%%%%%%%%%%%%%%%%%%%%%%%%%%%%%%%%%%%%
\begin{theorem}
\label{thm:padicpcompact}
%%%%%%%%%%%%%%%%%%%%%%%%%%%%%%%%%%%%%%%%%%%%%%%%%%%%%%%%%%
Let $X$ be a $p$-compact group, let
$M$ be a finite $\F_p[\pi_1 BX]$-module,
and let $N$ be a $\Z^\wedge_p[\pi_1 BX]$-module
which is finitely generated over $\Z^\wedge_p$. Then
\begin{enumerate}
\item the $ \Z^\wedge_p$-algebra $H^*(BX; \Z^\wedge_p)$ is Noetherian;

\item the module $H^*(BX; M)$ is Noetherian over $H^*(BX; \F_p)$;

\item the module $H^*(BX; N)$ is Noetherian over $H^*(BX; \Z^\wedge_p)$.
\end{enumerate}
\end{theorem}

\begin{proof}
The main theorem of Dwyer and Wilkerson, \cite[Theorem~2.4]{MR1274096},
asserts  that $H^*(BX;\F_p)$ is Noetherian. Lemma~\ref{lem:pcgbounded} allows us to apply
Theorem~\ref{thm:padicnoetherian} to prove the first claim. The second claim
follows then from Theorem~\ref{thm:Fpmodule} because $\pi_1 BX$
is a  finite $p$-group,
\cite[Lemma 2.1]{MR1274096}. Finally Theorem~\ref{thm:main} implies the third claim.
\end{proof}

%%%%%%%%%%%%%%%%%%%%%%%%%%%%%%%%%%%%%%%%%%%%%%%%%%%%%%%%%%
\begin{remark}
\label{rem:Cadek}
%%%%%%%%%%%%%%%%%%%%%%%%%%%%%%%%%%%%%%%%%%%%%%%%%%%%%%%%%%
{\rm Let us consider the case of $BO(n)$ at the prime $2$ (the
fundamental group is cyclic of order $2$). E.H. Brown made in \cite{MR652459}
an explicit computation of the integral cohomology. He actually proves that
the square of any even Stiefel-Whitney class $w_{2i}^2$ belongs to the
image of $\rho$ and the technique we use in Lemma~\ref{prop:finitetype2} is
somewhat inspired by his computations. Even though the
relations in the mod $p$ cohomology of an arbitrary $p$-compact
group (one which is not $p$-torsion free) make it difficult to
exhibit explicit generators for the $p$-adic cohomology,
Theorem~\ref{thm:padicpcompact} (1) gains qualitative control on it.
\par
As for twisted coefficients, let $\Z^\vee$ be a free abelian group of
rank one, endowed with  the sign action of the fundamental group $C_2$.
In \cite[Theorem~1]{MR1709293} \v{C}adek exhibits an
explicit \emph{finite} set of generators of $H^*(BO(n);
\Z^\vee)$, as a module over $H^*(BO(n); \Z)$. This is one of
the few available explicit computations illustrating our results.}
\end{remark}

\medskip

Broto, Levi, and Oliver defined in \cite{MR1992826} the concept of
$p$-local finite group. It consists of a triple $(S,\mathcal F,
\mathcal L)$ where $S$ is a finite $p$-group and, $\mathcal F$ and
$\mathcal L$ are two categories whose objects are subgroups of $S$.
The category $\mathcal F$ models abstract conjugacy relations among
the subgroups of $S$, and $\mathcal L$ is an extension of $\mathcal
F$ with enough information to define a classifying space $|\mathcal L|^\wedge_p$
which behaves like the $p$-completed classifying space of a finite group.
In fact, to any finite group $G$ corresponds a $p$-local finite group with 
$|\mathcal L|^\wedge_p \simeq (BG)^\wedge_p$, but there are also other ``exotic"
$p$-local finite groups.
%The classifying space of the $p$-local finite group is defined to be
%$|\mathcal L|^\wedge_p$.
%%%%%%%%%%%%%%%%%%%%%%%%%%%%%%%%%%%%%%%%%%%%%%%%%%%%%%%%%%
\begin{lemma}
\label{lem:plfgbounded}
%%%%%%%%%%%%%%%%%%%%%%%%%%%%%%%%%%%%%%%%%%%%%%%%%%%%%%%%%%
Let $(S,\mathcal F, \mathcal L)$ be a $p$-local finite group. The
$p$-torsion in $H^*(|\mathcal L|^\wedge_p; \Z^\wedge_p)$
is then bounded.
\end{lemma}
\begin{proof}
In  \cite[p. 815]{MR1992826} Broto, Levi and Oliver show the suspension spectrum  $\Sigma^\infty \left(|\mathcal L|^\wedge_p\right)$ is a retract of  $\Sigma^\infty BS$ following an idea due to Linckelmann and Webb (see also \cite{MR2199459}) . Since the order of $S$ annihilates all cohomology groups of $BS$, the same holds for
$H^*(|\mathcal L|^\wedge_p;\mathbb Z^\wedge_p)$.
\end{proof}
%%%%%%%%%%%%%%%%%%%%%%%%%%%%%%%%%%%%%%%%%%%%%%%%%%%%%%%%%%
\begin{theorem}
\label{thm:padicplfg}
%%%%%%%%%%%%%%%%%%%%%%%%%%%%%%%%%%%%%%%%%%%%%%%%%%%%%%%%%%
Let $(S,\mathcal F, \mathcal L)$ be a $p$-local finite group,
let $M$ be a finite $\F_p[\pi_1\left( |\mathcal L|^\wedge_p\right)]$-module, and
let $N$ be a $\Z^\wedge_p[\pi_1\left( |\mathcal L|^\wedge_p\right)]$-module
which is finitely generated over $\Z^\wedge_p$. Then
\begin{enumerate}
\item the $ \Z^\wedge_p$-algebra $H^*(|\mathcal L|^\wedge_p; \Z^\wedge_p)$ is Noetherian;

\item the module $H^*(|\mathcal L|^\wedge_p; M)$ is Noetherian over $H^*(|\mathcal L|^\wedge_p; \F_p)$;

\item the module $H^*(|\mathcal L|^\wedge_p; N)$ is Noetherian over $H^*(|\mathcal L|^\wedge_p; \Z^\wedge_p)$.
\end{enumerate}
\end{theorem}
\begin{proof}
We follow the same steps we took for $p$-compact groups in Theorem~\ref{thm:padicpcompact}.
The first ingredient is the stable elements theorem \cite[Theorem~B]{MR1992826}, which  also shows that $H^*(|\mathcal L|^\wedge_p; \F_p)$ is Noetherian. We just proved that the torsion in  $H^*(|\mathcal L|^\wedge_p; \Z^\wedge_p)$ is bounded.
Moreover, the fundamental group of $|\mathcal L|^\wedge_p$ is a finite $p$-group by \cite[Proposition~1.12]{MR1992826}.
\end{proof}

%%%%%%%%%%%%%%%%%%%%%%%%%%%%%%%%%%%%%%%%%%%%%%%%%
%%%%%%%%%%%%%%%%%%%%%%%%%%%%%%%%%%%%%%%%%%%%%%%%%
\appendix
\section{}%A few results on graded Noetherianity}
\label{appendix:Noether}
%%%%%%%%%%%%%%%%%%%%%%%%%%%%%%%%%%%%%%%%%%%%%%%%%
%%%%%%%%%%%%%%%%%%%%%%%%%%%%%%%%%%%%%%%%%%%%%%%%%
This short appendix deals with Noetherianity in the graded case over
the $p$-adics. We start however with a more general result, probably
well-known to the experts: the graded Eakin-Nagata Theorem. The
non-graded version can be found for example in Matsumura's book
\cite[Theorem~3.7(i)]{MR1011461}.

%%%%%%%%%%%%%%%%%%%%%%%%%%%%%%%%%%%%%%%%%%%%%%%%%%%%%%%%%%
\begin{proposition}
\label{gradedEakinNagata}
%%%%%%%%%%%%%%%%%%%%%%%%%%%%%%%%%%%%%%%%%%%%%%%%%%%%%%%%%%
Let $A^*$ be a graded subring of $B^*$. Assume that $B^*$ is
Noetherian as a ring and finitely generated as an $A^*$-module. Then
$A^*$ is also a Noetherian ring.
\end{proposition}

\begin{proof}
By \cite[Theorem 13.1]{MR1011461}, $B^0$ is Noetherian and $B^*$ is
a finitely generated $B^0$-algebra. Moreover, $B^0$ is a finitely
generated $A^0$-module and therefore $B^*$ is a finitely generated
$A^0$-algebra. Also, $A^0$ is Noetherian by the classical
Eakin-Nagata theorem \cite[Theorem~3.7(i)]{MR1011461}. Applying
\cite[Proposition 7.8]{MR0242802} to the inclusions $A^0\subset
A^*\subset B^*$ we obtain that $A^*$ is a finitely generated
$A^0$-algebra. Again, by \cite[Theorem 13.1]{MR1011461}, $A^*$ is a
Noetherian ring.
\end{proof}

The following technical proposition  allows us to deduce Noetherianity over the $p$-adics from
the Noetherianity of the mod $p$ reduction.

%%%%%%%%%%%%%%%%%%%%%%%%%%%%%%%%%%%%%%%%%%%%%%%%%%%%%%%%%%
\begin{proposition}
\label{prop:Hausdorff}
%%%%%%%%%%%%%%%%%%%%%%%%%%%%%%%%%%%%%%%%%%%%%%%%%%%%%%%%%%
Let $A^*$ be a graded $\Z^\wedge_p$-algebra such
that in each degree $A^k$ is complete for the $p$-adic topology. Let
$N^*$ be a graded $A^*$-module such that for all $k$, $N^k$ is
Hausdorff for the $p$-adic topology. If $N^* \otimes \F_p$ is
a Noetherian $A^*$-module, then so is $N^*$.
\end{proposition}

\begin{proof}
Let us choose homogeneous elements $\nu_1,...,\nu_t \in N^*$ such that $\nu_1
\otimes 1,...,\nu_t \otimes 1$ generate $N^* \otimes \F_p$ as
$A^*$-module. We claim that $\nu_1,...,\nu_t$ generate $N^*$ as an
$A^*$-module. Given $n \in N^*$ we may write
$n \otimes 1 = \sum a_i^0 (\nu_i \otimes 1)$ for some $a_i^0 \in A^*$.
Define $n_0 =\sum a_i^0 \nu_i$ and notice that $n -n_0 \in p N^*$.
Thus, there exists an element $m_1 \in N^*$, homogeneous of degree $\leq \deg n$,
such that $n-n_0 = p m_1$.
We iterate the procedure and find elements $a_i^1 \in A^*$ such that
$m_1 \otimes 1 = \sum a_i^1 (\nu_i \otimes 1)$.
We define $n_1 = n_0 + p \sum a_i^1 \nu_i = \sum
(a_i^0 + p a_i^1) \nu_i$. In this way we construct, for any $i$,
Cauchy sequences of coefficients $(a_i^0 + p a_i^1 + \dots + p^k
a_i^k)_k$ in $A^*$. By completeness this sequence converges to some
$a_i \in A^*$. Since $N^*$ is Hausdorff, the element $\sum a_i \nu_i$
is equal to $n$.
\end{proof}
In the following corollary, the assumption that $A^*$ be connected, i.e. $A^0 = \Z^\wedge_p$,
is important. 
%The referee kindly provided a non-connected counter-example, namely $A^* = \Q^\wedge_p(t)$, which is Noetherian, but not finitely generated as a $\Z^\wedge_p$-algebra.
%%%%%%%%%%%%%%%%%%%%%%%%%%%%%%%%%%%%%%%%%%%%%%%%%%%%%%%%%%
\begin{corollary}
\label{cor:Hausdorff-algebra}
%%%%%%%%%%%%%%%%%%%%%%%%%%%%%%%%%%%%%%%%%%%%%%%%%%%%%%%%%%
Let $A^*$ be a graded connected Hausdorff $\Z^\wedge_p$-algebra.
If $A^*\otimes \F_p$ is a Noetherian $\F_p$-algebra,
then $A^*$ is a Noetherian $\Z^\wedge_p$-algebra.
\end{corollary}

\begin{proof}
Since $\Z^\wedge_p$ is Noetherian and $A^*$ is connected, $A^*$ is a Noetherian
$\Z^\wedge_p$-algebra if and only if $A^*$ is a finitely
generated $\Z^\wedge_p$-algebra, \cite[Theorem 13.1]{MR1011461}. Note that $A^* \otimes
\F_p$ is also a Noetherian $\Z^\wedge_p$-algebra via
the mod $p$ reduction $\Z^\wedge_p\to \F_p$. Let us
choose homogeneous elements $\gamma_1,...,\gamma_n \in A^*$ such that $\gamma_1
\otimes 1,...,\gamma_n \otimes 1$ generate $A^* \otimes \F_p$
as a $\Z^\wedge_p$-algebra. For a fixed $k \geq 0$, $A^k \otimes \F_p$ is
generated as a $\Z^\wedge_p$-module by the monomials
$(\gamma_1\otimes 1)^{e_1}\cdots (\gamma_n\otimes 1)^{e_n}$ with
$\sum_{i=1}^n|\gamma_i|e_i= k$. Since $A^k$ is a Hausdorff
$\Z^\wedge_p$-module, the proof of Proposition \ref{prop:Hausdorff}
shows that $A^k$ is generated by the monomials
$\gamma_1^{e_1}\cdots \gamma_n^{e_n}$ with
$\sum_{i=1}^n|\gamma_i|e_i= k$.
This shows that $A^*$ is generated as a
$\Z^\wedge_p$-algebra by the elements $\gamma_1,...,\gamma_n
\in A^*$ and therefore $A^*$ is a Noetherian $\Z^\wedge_p$-algebra.
\end{proof}

%%%%%%%%%%%%%%%%%%%%%%%%%%%%%%%%%%%%%%%%%%%%%%%%%%
%%%%%%%%%%%%%%%%%%% REFERENCES %%%%%%%%%%%%%%%%%%%
%%%%%%%%%%%%%%%%%%%%%%%%%%%%%%%%%%%%%%%%%%%%%%%%%%

\bibliographystyle{amsplain}
\providecommand{\bysame}{\leavevmode\hbox to3em{\hrulefill}\thinspace}
\providecommand{\MR}{\relax\ifhmode\unskip\space\fi MR }
\providecommand{\MRhref}[2]{%
  \href{http://www.ams.org/mathscinet-getitem?mr=#1}{#2}
}
\providecommand{\href}[2]{#2}

%%%%%%%%%%%%%%%%%%%%%%%%%%%%%%%%%%%%%%

%%%%%%%%%%%%%%%%%%%%%%%%%%%%%%%%%%%%%%%%%%%%%%%%%
%%%%%%%%%%%%%%%%%%% ADDRESSES %%%%%%%%%%%%%%%%%%%
%%%%%%%%%%%%%%%%%%%%%%%%%%%%%%%%%%%%%%%%%%%%%%%%%
\bigskip
{\small
\begin{minipage}[t]{8 cm}
Kasper K. S. Andersen\\ 
Centre for Mathematical Sciences\\
LTH, Box 118\\
SE-22100 Lund, Sweden\\
\textit{E-mail:}\texttt{\,kksa@maths.lth.se}\\
\\
Vincent Franjou\\
Laboratoire Jean-Leray - UMR 6629\\
2, rue de la Houssini\`ere\\
BP 92208\\
F-44322 Nantes cedex 3, France\\
\textit{E-mail:}\texttt{\,Vincent.Franjou@univ-nantes.fr}\\
\\
J\'er\^ome Scherer\\
EPFL SB MATHGEOM\\
MA B3 455\\
Station 8\\
CH -1015 Lausanne, Switzerland\\
\textit{E-mail:}\texttt{\,jerome.scherer@epfl.ch}
\end{minipage}
\begin{minipage}[t]{8 cm}
Nat\`{a}lia Castellana\\ 
Departament de Matem\`atiques\\
Universitat Aut\`onoma de Barcelona\\ 
E-08193 Bellaterra, Spain\\
\textit{E-mail:}\texttt{\,natalia@mat.uab.cat}\\
\\
Alain Jeanneret\\
Mathematisches Institut\\
Universit\"at Bern\\
Sidlerstrasse 5\\
CH-3012 Bern, Switzerland\\
\textit{E-mail:}\texttt{\,alain.jeanneret@math.unibe.ch}\\
\end{minipage}

%%%%%%%%%%%%%%%%%%%%%%%%%%%%%%%%%%%%%%
\end{document}